\documentclass[11pt]{amsart}
\usepackage{amsmath,amsthm, amscd, amssymb, amsfonts, mathtools,color}
\usepackage[all]{xy}
\usepackage{mathrsfs}
\usepackage{enumitem}
\usepackage[T2A,T1]{fontenc}
\usepackage{multicol}

\usepackage{hyperref}

\allowdisplaybreaks

\usepackage[ansinew]{inputenc}
\usepackage{graphicx,fancyhdr}

\newcommand{\Ext}{\operatorname{Ext}}
\newcommand{\imm}{\operatorname{Im}}

\newcommand{\wb}{\mathbf w}
\newcommand{\Nuc}{\mathcal O}

\newcommand{\Gb}{\mathbf G}

\newcommand{\supp}{\operatorname{supp}}
\newcommand{\verma}[1]{\mathtt{M}(#1)}
\newcommand{\vermaopp}[1]{\overline{\mathtt{M}}(#1)}

\newcommand{\smp}[1]{\mathtt{L}(#1)}
\newcommand{\vermasub}[1]{\mathtt{N}(#1)}
\newcommand{\usrl}[1]{\mathtt{T}(#1)}
\newcommand{\usrldos}[2]{\mathtt{S}_{#2}(#1)}
\newcommand{\indescdos}[3]{\mathtt{P}_{#2 , #3}(#1)}
\newcommand{\vermados}[3]{\mathtt{M}_{#2, #3}(#1)}
\newcommand{\Mtt}{\mathtt{M}}

\newtheorem{theorem}{Theorem}[section]
\newtheorem{lemma}[theorem]{Lemma}

\newtheorem{prop}[theorem]{Proposition}

\theoremstyle{definition}
\newtheorem{definition}[theorem]{Definition}

\newtheorem*{claim}{Claim}
\newtheorem{case}{Case}

\theoremstyle{remark}
\newtheorem{remark}[theorem]{Remark}

\newtheorem{step}{Claim}

\newcommand{\pf}{\begin{proof}}
\newcommand{\epf}{\end{proof}}

\newcommand{\spl}{\mathfrak{sl}}

\newcommand{\ku}{ \Bbbk}

\newcommand{\kut}{ \ku^{\times}}

\newcommand{\stt}{\mathtt{s}}
\newcommand{\wtt}{\mathtt{w}}

\newcommand{\I}{\mathbb I}

\newcommand{\N}{\mathbb N}

\newcommand{\Z}{\mathbb Z}
\newcommand{\zt}{\xi}
\newcommand{\zn}[2]{z_{#1, #2}}
\newcommand{\znp}[2]{z_{(#1, #2)}}
\newcommand{\tzn}[2]{\mathtt{z}_{(#1, #2)}}

\newcommand{\tz}{\mathtt{z}}

\newcommand{\cO}{\mathcal{O}}
\newcommand{\cE}{\mathcal{E}}

\newcommand{\cC}{\mathcal{C}}

\newcommand{\D}{\mathcal{D}}

\newcommand{\Pc}{{\mathcal P}}

\newcommand{\Ss}{{\mathcal S}}

\newcommand{\End}{\operatorname{End}}
\newcommand{\id}{\operatorname{id}}

\newcommand{\GK}{\operatorname{GKdim}}

\newcommand{\Irr}{\operatorname{irrep}}
\newcommand{\IRR}{\operatorname{Irrep}}
\newcommand{\Ind}{\operatorname{Ind}}
\newcommand{\Res}{\operatorname{Res}}
\newcommand{\ot}{\otimes}

\newcommand{\yd}[1]{{}^{ #1 }_{ #1 }\mathcal{YD}}

\newcommand{\Lmod}[1]{{}_{ #1 \hspace{-2pt}}\operatorname{Mod}}
\newcommand{\lmod}[1]{{}_{ #1 \hspace{-2pt}}\operatorname{mod}}
\newcommand{\grLmod}[1]{{}_{ #1 }\mathcal{SG}}
\newcommand{\hwmod}{\mathfrak O}
\newcommand{\hw}{\operatorname{hw}}
\newcommand{\hwrank}{\operatorname{hw-rk}}
\newcommand{\lwmod}{\mathfrak o}
\newcommand{\lw}{\operatorname{lw}}
\newcommand{\lwrank}{\operatorname{lw-rank}}
\newcommand{\dual}{\vee}

\newcounter{tabla}\stepcounter{tabla}

\begin{document}

\title[Representations of the double of the Jordan plane]
{On the finite-dimensional representations of the double of the Jordan plane}

\author[Andruskiewitsch]{Nicol\'as Andruskiewitsch}
\address[N. Andruskiewitsch]{Facultad de Matem\'atica, Astronom\'ia y F\'isica,
Universidad Nacional de C\'ordoba. CIEM -- CONICET. 
Medina Allende s/n (5000) Ciudad Universitaria, C\'ordoba, Argentina}
\email{nicolas.andruskiewitsch@unc.edu.ar}

\author[Pe\~na Pollastri]{H\'ector Mart\'\i n Pe\~na Pollastri}
\address[H.~Pe\~na Pollastri]{On leave from Facultad de Matem\'atica, Astronom\'ia y F\'isica, Universidad Nacional de C\'ordoba. CIEM -- CONICET. 
Medina Allende s/n (5000) Ciudad Universitaria, C\'ordoba, Argentina}
\email{hpollastri@famaf.unc.edu.ar}

\address{\hspace{-20pt} Department of Mathematics, Indiana University, Bloomington, IN 47405 USA}
\email{hpenapol@iu.edu}

\thanks{N. A. and H. M. P. P. were partially supported by CONICET (PIP  11220200102916CO),
FONCyT-ANPCyT (PICT-2019-03660) and Secyt (UNC). H. M. P. P. was also supported by the Simons Foundation Award 889000 as part of the Simons Collaboration on Global Categorical Symmetries.}


\begin{abstract}
We continue the study of the  
Drinfeld double  of the Jordan plane, denoted by $\mathcal D$ and introduced in \cite{ap}. 
 The  simple finite-dimensional modules  were computed in \cite{adp};  it turns out that they factorize through $U(\spl_2(\Bbbk))$. 
Here we introduce the Verma modules and  the category $\mathfrak O$ for $\mathcal D$, 
 which have a resemblance to the similar ones  in Lie theory   but induced from indecomposable modules
 of  the 0-part of the triangular decomposition. 
 Accordingly, there is the notion of highest weight rank (hw-rk). 
 We classify the indecomposable modules of hw-rk  one and find families of hw-rk two. 
 The Gabriel quiver of $\mathcal D$ is  computed implying that it has a wild representation type.
\end{abstract}

\maketitle
\setcounter{tocdepth}{1}
\tableofcontents

\section{Introduction}

Let $\ku$ be an algebraically closed field of characteristic $0$.
There were substantial advances in the ongoing classification of Nichols algebras over abelian groups with finite Gelfand-Kirillov dimension (abbreviated as $\GK$), an important step towards the classification of Hopf algebras with finite $\GK$, see  \cite{aah-triang,aam,angiono-garcia}. Recall that Nichols algebras of diagonal type  correspond to 
infinitesimal braidings that are semisimple as Yetter-Drinfeld modules (over abelian groups). 
In the context of finite $\GK$, Nichols algebras  
with non-semisimple infinitesimal braidings appear naturally. 
The well-known Jordan plane, denoted by $J$,  is a paradigmatic example.

\medbreak
Quantized enveloping algebras can be defined via Drinfeld doubles 
of (bosonizations of) suitable Nichols algebras of diagonal type. 
A natural question is to investigate the behaviour of the Drinfeld doubles 
of other Nichols algebras, specifically those with non-semisimple infinitesimal braidings. 
In this direction,  we studied in \cite{ap,adp} a Hopf algebra $\D$ 
thought of as the Drinfeld double of the bosonization $J \# \ku \Z$, which differs from quantized enveloping algebras
for some distinctive features:

 \medbreak
\begin{itemize}[leftmargin=*]\renewcommand{\labelitemi}{$\circ$}
\item\cite{ap}  The Hopf algebra paired with  $J \# \ku \Z$ in order to build the double 
has the form $J^d \# U(\mathfrak h)$
where $J^d$ is the graded dual of  $J$ and $\dim \mathfrak h = 1$.

 \medbreak
\item \cite{ap}
The Hopf algebra $\D$ is not simple and, in fact, it
fits into an \emph{abelian} exact sequence of Hopf algebras 
$\xymatrix{ \cO(\Gb)  \ar@{^{(}->}[r]  & \D \ar@{->>}^{\pi\qquad}[r]  & U(\spl_2(\ku)) }$
where $\cO(\Gb)$ is the algebra of regular functions on $\Gb = (\Gb_a \times \Gb_a) \rtimes \Gb_m$.

 \medbreak
\item \cite{ap}
The analogue of the Cartan subalgebra of $\D$ splits as $\D^0 \simeq U(\mathfrak h) \otimes \ku \Gamma$. 

 \medbreak
\item \cite{adp} 
The finite-dimensional simple modules over $\D$ are in bijective correspondence with those of $U(\spl_2(\ku))$ via $\pi$. 
\end{itemize} 

 \medbreak
The next  step is to understand the indecomposable modules. 
We observe that the action of $\D^0$ is not semisimple, what leads us to introduce 
Verma modules inducing from indecomposable $\D^0$-modules; highest weight modules
and the category $\hwmod$   they belong to. 
The notions of $\hwrank$ and hw-series arise naturally. 
See Section \ref{sec:highest-weight-modules}.
 In Sections \ref{sec:hwr1} and \ref{sec:hwr2} 
 we study the indecomposable modules with $\hwrank$ one and two respectively. 
 In the case of $\hwrank = 1$ we get the complete classification by the family of uniserial indecomposable modules $\usrl{n,m}$, see Theorem \ref{thm:indecomposables-rank-1-class}. 
In Section \ref{sec:extensions-simple-modules} we classify extensions of simple modules
using the families  described previously. 
Thus we compute the Gabriel quiver of $\D$ and deduce that $\D$ has wild representation type, see Proposition \ref{prop:D-has-wild-rep-type}.
Section \ref{sec:preliminaries} is devoted to preliminaries on the algebra $\D$.

\medbreak
There are convincing reasons to believe that the representation theory of the Drinfeld doubles of many of the Nichols algebras discovered in \cite{aah-triang,aam} will have the features described above.

\subsection*{Conventions}
If $\ell < n \in\N_0$, we set $\I_{\ell, n}=\{\ell, \ell +1,\dots,n\}$, $\I_n = \I_{1, n}$. 
If $Y$ is a subobject of an object $X$ in a category $\cC$, then we write $X\leq Y$.

Let $A$ be an algebra. 
Given $a_1,\dots,a_n\in A$, $n\in\N$, $\ku\langle a_1,\dots,a_n\rangle$ denotes the subalgebra generated by $a_1,\dots,a_n$.
Let $\Lmod{A}$
(respectively $\lmod{A}$, $\Irr A$)  denote  the category of left $A$-modules
(respectively, the full subcategory of finite-dimensional ones, the set of isomorphism classes of  simple objects  in $\lmod{A}$).
Often we do not  distinguish a class in $\IRR A$ and one of its representatives. 
If $M \in \Lmod{A}$ and $m_1,\dots,m_n\in M$, $n\in\N$, then $\langle m_1,\dots,m_n\rangle$ 
denotes the submodule generated by $m_1,\dots,m_n$.
Given a subalgebra $B$, of $A$, $\Ind_{B}^{A}: \Lmod{B} \to \Lmod{A}$ and $\Res^{B}_A: \Lmod{A} \to \Lmod{B}$ denote the induction  and restriction functors, e.g. $\Ind_{B}^{A}(M) =  A \otimes_B M$. 

Let $L$ be a Hopf algebra. The kernel of the counit $\varepsilon$ is denoted $L^+$, the antipode (always assumed bijective) by $\Ss$, the
space of primitive elements by $\Pc(L)$ and the group of group-likes by $G(L)$.
The space of $(g,h$)-primitives is $\Pc_{g,h}(L) =\{x \in L: \Delta(x) = x\otimes h + g \otimes x\}$ where $g,h \in G(L)$.
The category of Yetter-Drinfeld modules over $L$ is denoted by $\yd{L}$.
We refer to \cite{Rad-libro} for unexplained terminology on Hopf algebras.

\section{Preliminaries}\label{sec:preliminaries}
\subsection{The double of the Jordan plane}\label{subsec:double}
The well-known Jordan plane, the quadratic algebra $J = \ku\langle x,y \vert xy - yx - \frac{1}{2}x^2\rangle$,
is a braided Hopf algebra where $x$ and $y$ are primitive. 
The Hopf algebra $\D$ was introduced in \cite[Proposition 2.3]{ap}, where it is denoted $\widetilde{D}$.
See \cite{ap,adp} for properties of $\D$, some of which are  listed below.
The algebra $\D$ 
is presented by generators $u$, $v$, $\zt$, $g^{\pm 1}$, $x$, $y$ and relations
\begin{align}
\label{eq:def-jordan1}
&\begin{aligned}
g^{\pm 1} g^{\mp 1}&= 1, &\zt  g &= g \zt, 
\end{aligned}
\\ \label{eq:def-jordan2}
&\begin{aligned}
gx &= xg, & gy &= yg +xg, & \zt  y &= y \zt  -2 y, &  \zt  x&=x \zt  - 2 x, 
\\
u g &= gu, & v g &=gv  + gu, & v\zt  &= \zt v  -2 v, &  u\zt  &= \zt u -2 u,
\end{aligned}
\\
\label{eq:def-jordan3}
&\begin{aligned}
yx &= xy - \frac{1}{2} x^2, & vu &= uv - \frac{1}{2}u^2,
\\
u x &=x u , &   v x &= x v  + (1-g) + xu, 
\\
u y &=yu  +(1 - g), &  v y &=y v + \frac{1}{2}g \zt  + y u. 
\end{aligned}
\end{align}
Notice that we have replaced the generator $\zeta$ by $\zt = -2\zeta$ in the presentation of \cite{ap} and adjusted the relations accordingly.
The  Hopf algebra structure is determined by 
$g\in G(\D)$, $u,\zt \in\Pc(\D)$, $x,y\in\Pc_{g,1}(\D)$ and 
\begin{align*}
\Delta( v) &=  v\ot 1 + 1\ot v -\frac{1}{2}\zt\ot u. 
\end{align*}

The following set  is a PBW-basis of $\D$:
\begin{align*}
B &= \{x^n\,y^r\,g^m\,\zt ^k\,u^i\,v^j: i,j,k,n,r\in\N_0, \quad m\in\Z\}.
\end{align*}

We shall consider the  following  subalgebras    of $\D$:
\begin{align*}
\D^{<0} &\coloneqq \ku \langle x, y\rangle, & 
\D^{0} &\coloneqq \ku \langle g^{\pm1}, \zt\rangle, &
 \D^{\leq0} &\coloneqq \ku\langle g^{\pm1}, \zt, x,y \rangle,
 \\ 
 \D^{>0} &\coloneqq \ku \langle u, v\rangle & &\text{and} &
\D^{\geq0} &\coloneqq \ku\langle g^{\pm1}, \zt, u, v\rangle.
\end{align*}

The algebra $\D$ has a  $\Z$-grading  $\D = \oplus_{n\in\Z}\D^{[n]}$ given by 
\begin{align}\label{eq:D-grading}
\deg x =\deg y = -2, && \deg u =\deg v = 2, && \deg g = \deg \zt = 0.
\end{align}

The algebra $\D$ has a triangular decomposition 
$\D \simeq \D^{<0} \otimes \D^0 \otimes \D^{>0}$.

\subsection{An exact sequence of Hopf algebras}
Let $\Nuc\coloneqq \ku\langle x,u,g^{\pm 1}\rangle$; this is a commutative Hopf subalgebra of $\D$, 
hence $\Nuc \simeq \cO(\Gb)$, 
where $\Gb$ is the algebraic group as in the Introduction.  
Let $e, f, h$  be  the Chevalley generators of  $\spl_2(\ku)$, i.~e.
$[e,f]=h$,  $[h,e]=2e$, $[h,f]=-2f$.
The Hopf algebra map $\pi \colon \D \to U(\spl_2(\ku))$ determined by
\begin{align}\label{eq:iso-hopf}
\pi (v)&= \tfrac{1}{2} e,& \pi(y) &= f  ,& \pi (\zt)&=  h,&
\pi (u)&= \pi(x) =  \pi (g-1) =0,
\end{align}
induces  an
isomorphism of Hopf algebras $\D/\D\Nuc^+ \simeq U(\spl_2(\ku))$.   
Thus we have an exact sequence as mentioned in the Introduction.

\subsection{Simple modules}\label{subsec:verma-simple}
The simple objets in $\lmod{\D}$ are classified in \cite{adp}; the starting point is
the following fact that we will use later.

\begin{prop}\label{prop:g-1-x-u-D-nilpotent} \cite[Proposition 3.9]{adp}
If $M\in \lmod{\D}$, then  $g-1$, $x$ and $u$ act  nilpotently on $M$. \qed
\end{prop}
Let $\smp{n}$ be the simple $\spl_2(\ku)$-module with highest weight $n$.
Then $\smp{n}$ becomes a simple $\D$-module  via the projection $\pi$ in \eqref{eq:iso-hopf}. Precisely,  
$\smp{n}$ has a basis $\{\zn{n}{0},\dots,\zn{n}{n}\}$   where the action is given by 
\begin{align}\label{eq:D-action-Ln}
\begin{aligned}
\zt\cdot \zn{n}{i} &=  (n-2i) \zn{n}{i}, & v \cdot \zn{n}{i} &= \tfrac{i (n-i+1)}{2}  \zn{n}{i-1}, & y \cdot \zn{n}{i} &= \zn{n}{i+1},
\\
x \cdot \zn{n}{i} &= 0, & u \cdot \zn{n}{i} &= 0,   & g\cdot \zn{n}{i} &= \zn{n}{i},
\end{aligned}
\end{align}
$i\in \I_{0,n}$ where $\zn{n}{-1} = \zn{n}{n+1} = 0$ by convention.

\begin{theorem}\label{th:Ln-simples} \cite[3.11]{adp}
The family $(\smp{n})_{n\in\N_0}$ 
parametrizes  $\Irr \D$. \qed
\end{theorem}

\begin{remark}\label{rem:locally-nilpotent}
If $M\in \lmod{\D}$, then   $y$ and $v$ act  nilpotently on $M$ and the eigenvalues of the action of $\zt$ are integers, but the action of $\zt$ is not necessarily semisimple
(consider a Jordan-H\"older series of $M$ and apply Theorem \ref{th:Ln-simples}).
\end{remark}

\section{Highest weight modules}\label{sec:highest-weight-modules}

 \subsection{Weight decompositions} Proposition \ref{prop:g-1-x-u-D-nilpotent}
 and Remark \ref{rem:locally-nilpotent}   lead us to the following considerations. 
 Given $P\in \Lmod{\D^0}$ and $n\in \Z$, we set  
 \begin{align*}
 P^{(n)} &= \{m\in P:  (\zt  - n)^a \cdot m = 0, \ (g -1)^b \cdot m = 0 \text{ for some } a,b\in \N_0 \}.
 \end{align*}
 By a standard argument, see e.~g. \cite[1.2.13]{dix}, the sum $\sum_{n\in \Z} P^{(n)}$ is direct.
 Let  $\grLmod{\D^0}$ be the  full subcategory of $\Lmod{\D^0}$ consisting of those
 $P\in \Lmod{\D^0}$  such that $P = \oplus_{n\in \Z} P^{(n)}$.
 Notice that $\lmod{\D_0}$ is a subcategory of $\grLmod{\D^0}$.
 
 \medbreak
 Given $M \in \Lmod{\D}$, we set  by abuse of notation
 \begin{align*}
 M^{(n)} =  \big(\Res ^{\D^0}_{\D} (M) \big)^{(n)}.
 \end{align*}  
 The  weights of $M \in \Lmod{\D}$ are the elements of  
$\varPi(M) \coloneqq \{n\in\Z: M^{(n)} \neq 0\}$;
 the $M^{(n)}$'s are called weight subspaces even when they are 0.
 
 \begin{definition}
 A module $M \in \Lmod{\D}$ is \emph{suitably graded} if
 \begin{align*}
 M &=\oplus_{n\in\Z} M^{(n)} && \text{and} & \dim M^{(n)} &< \infty & \text{for all } n&\in\Z.
 \end{align*}
 Let $\grLmod{\D}$ be the  full subcategory of $\Lmod{\D}$  of suitably graded modules.
 \end{definition}  
 
 \begin{remark}
 \begin{enumerate}[leftmargin=*,label=\rm{(\roman*)}]
 \item 
 From the defining relations we get that 
 \begin{align}\label{eq:weight-action}
 \begin{aligned}
 x\cdot M^{(n)} &\subset M^{(n-2)} \supset y\cdot M^{(n)}, & g\cdot M^{(n)} &= M^{(n)} \supset \zt \cdot M^{(n)}, \\
 u\cdot M^{(n)} &\subset M^{(n+2)} \supset v\cdot M^{(n)}, & n &\in \Z.
 \end{aligned}
 \end{align}
 
 \medbreak
 \item Morphisms of $\D$-modules preserve the weight subspaces. 
Thus submodules and quotients of suitably graded modules are suitably graded.
Given an exact sequence of  $\D$-modules  $\xymatrix@C=0.5cm{ N \ar@{^{(}->}[r]   & M  \ar@{->>}[r] & S}$,
the sequence $\xymatrix@C=0.5cm{ N ^{(n)}\ar@{^{(}->}[r]   & M^{(n)}  \ar@{->>}[r] & S^{(n)}}$ of  $\D^0$-modules
 is exact   for any $n\in \Z$.

 \medbreak
 \item\label{item:iii} Any $M \in \lmod{\D}$ is suitable graded by Remark \ref{rem:locally-nilpotent}.
 If $n \in \varPi(M)$,  then 
 \begin{align*}
 \dim M^{(n)} &= \dim M^{(-n)}& &\text{ and }& \{n, n-2, n-4, \dots, 2-n, -n\} &\subseteq \varPi(M).
 \end{align*}
 \end{enumerate}
 \end{remark}
 
 \pf  Let $\lambda \in \ku$. Then $(\zt - \lambda + 2)y = y(\zt - \lambda)$ by \eqref{eq:def-jordan2}; hence  by induction
 $(\zt - \lambda + 2)^ay = y(\zt - \lambda)^a$ for any $a\in \N$. 
 Also, $(g-1)y = y(g-1) + gx$, hence  by induction
 $(g-1)^ay = y(g-1)^a +  axg(g-1)^{a-1}$ for any $a\in \N$. 
 Therefore 
 $y\cdot M^{(n)} \subset M^{(n-2)}$; the rest is similar.
 The proof of \ref{item:iii} follows from the representation theory of $\spl_2$.
 \epf
 
 Let $M\in \grLmod{\D}$. Given $n\in \Z$, one identifies the dual $\big(M^{(n)}\big)^*$
with the subspace  $\big(\oplus_{n \neq a\in\Z} M^{(a)} \big)^{\perp} = \{f\in M^*: f_{\vert M(a)} = 0, a \neq n\}$
of $M^*$.   Clearly the sum of the various duals $\big(M^{(n)}\big)^*$ is direct. The graded dual of $M$ is
\begin{align*}
M^{\dual}  &= \oplus_{n\in\Z} \big(M^{(n)}\big)^* \hookrightarrow M^*.
\end{align*} 
 We need the formula for the antipode:
\begin{align}\label{eq:antipode}
\begin{aligned}
\Ss(g) &= g^{-1}, & \Ss(x) &= -g^{-1}x, & \Ss(y) &= -g^{-1}y, \\
\Ss(\xi) &= -\xi, & \Ss(u) &= -u, &\Ss(v) &= -v -\frac{1}{2} \xi u.
\end{aligned}
\end{align}

 \begin{lemma}\label{lema:graded-dual}
 If $M  \in \grLmod{\D}$, then $M^{\dual}\in \grLmod{\D}$
 and 
 \begin{align}\label{eq:graded-dual}
(M^{\dual})^{(n)} &= \big(M^{(-n)}\big)^*,& n&\in \Z.
 \end{align}
 \end{lemma}
 
\pf  Fix $q \in \Z$. Notice that $\big(M^{(q)}\big)^*$ is stable under the action of $g$ and $\zt$.
We  claim that  
\begin{align*} 
\begin{aligned}
x\cdot \big(M^{(q)}\big)^* &\subset \big(M^{(q+2)}\big)^* \supset y\cdot \big(M^{(q)}\big)^*, \\
u\cdot \big(M^{(q)}\big)^* &\subset \big(M^{(q-2)}\big)^* \supset v\cdot \big(M^{(q)}\big)^*.
\end{aligned}
\end{align*}
Indeed, if $f\in \big(M^{(q)}\big)^*$, $p\in \Z$ and $m \in M^{(p)}$, then
\begin{align*}
\langle x\cdot f, m \rangle &= -\langle f, x\cdot m \rangle  = 0 \text{ unless } x\cdot m\in M^{(q)}
\text{ that is } p- 2 = q.
\end{align*}
The other inclusions are similar. Thus $M^{\dual}$ is a submodule of $M^*$.
We next claim that $\big(M^{(q)}\big)^* \subset \big(M^*\big)^{(-q)}$.
Since $\dim M^{(q)} < \infty$,  
there exist $a,b \in \N$ such that  $(\zt - q)^a$  and $(g - 1)^b$ act by 0 on $M^{(q)}$. 
Pick $f\in \big(M^{(q)}\big)^*$. By \eqref{eq:antipode} we have, for any $p\in \Z$ and $m \in M^{(p)}$, that
\begin{align*}
\langle(\zt + q)\cdot f, m \rangle &= \langle f, (-\zt +q)\cdot m \rangle \\
\implies  \langle(\zt +q )^a\cdot f, m \rangle 
&= (-1)^a\langle f, (\zt - q)^a\cdot m \rangle =0;
\\
\langle(g -1)\cdot f, m \rangle &= \langle f, (g^{-1} - 1)\cdot m \rangle \\
\implies  \langle(g - 1)^b\cdot f, m \rangle 
&= (-1)^b\langle f, (g - 1)^b \cdot (g^{-b} \cdot m) \rangle =0
\end{align*}
The claim follows. Together with the first claim, this implies \eqref{eq:graded-dual}.
 \epf
 
 \begin{remark} \label{rem:tensor-product}
 If $M,N  \in \grLmod{\D}$, then $M\otimes N =\oplus_{n\in\Z} (M\otimes N)^{(n)}$ and 
\begin{align*}
 (M\otimes N)^{(n)} &= \oplus_{p+ q =n} M^{(p)}\otimes N^{(q)}.
\end{align*}
 \end{remark}
Notice however that  $\dim (M\otimes N)^{(p+q)}$ is not necessarily finite. 
 
 \medbreak
 \noindent\emph{Proof. } 
 It suffices to show that $M^{(p)}\otimes N^{(q)} \subset (M\otimes N)^{(p+q)}$ for any  $p,q\in \Z$; this follows from
 the equalities
 \begin{align*}
 \Delta (\zt - \lambda)^a &= \sum_{0\leq c \leq a} \binom{a}{c} (\zt - \mu)^c \otimes (\zt - \nu)^{a-c}, & \mu+ \nu = \lambda;
 \\
 \Delta (g - 1)^b &= \sum_{0\leq d \leq b} \binom{b}{d} (g-1)^d \otimes g^d(g - 1)^{b-d}, &a,b &\in \N. &&\qed
 \end{align*}

 \subsection{Highest weight modules}
 
Let us now  consider $P\in \lmod{\D^0}$  such that $P = P^{(n)}$ for some $n\in \Z$. 
 Then $P$ becomes a $\D^{\geq0}$-module by $u\cdot m = 0$, $v\cdot m = 0$, $m \in P$.
 We define the  Verma module
 \begin{align*}
 \verma{P} \coloneqq \Ind_{\D^{\geq0}}^{\D} (P) \simeq \D\ot_{\D^{\geq0}} P.
 \end{align*} 
 Thus $\verma{P}  \simeq \D^{<0}\ot_{\ku} P$ as  vector spaces\footnote{Notice that the Verma modules in \cite{adp} are induced from one-dimensional modules where $g$ and $\zt$ act by arbitrary eigenvalues, thus more and less general than the previous definition.}.

 \begin{remark}\label{rem:verma-module-graded-detailed}
 The  Verma module $\verma{P}$ is suitable graded. Indeed we have $P = P^{(n)} \subset \verma{P}^{(n)} $
 by definition, hence $y^ix^jP  \subset \verma{P}^{(n - 2(i+j))} $ for any $i,j \in \N_0$ by \eqref{eq:weight-action}.
 That is, $\verma{P}^{(n - 2k)} = \oplus_{i+j =k}y^ix^jP $ and 
 \begin{align*}
 \verma{P} = \bigoplus_{k\in \N_{0}} \verma{P}^{(n - 2k)}.
 \end{align*}
 \end{remark}
 
 \begin{definition}  Let $\hwmod$ be the  full subcategory of $\grLmod{\D}$  consisting of 
those  $M$ such that $ \varPi(M) $ is bounded above.  For $M \in \hwmod$, 
its highest weight is $\hw M := \supp \varPi(M)$; we introduce
$\hwrank M = \dim M^{(\hw M)}$.
 \end{definition} 
 
  \begin{definition}\label{def:hwmodule}
 If $M \in \hwmod$  is generated by $M^{(\hw M)}$, then 
 we say that $M$ is a \emph{highest weight}  module.
  \end{definition} 
 
 \begin{remark}
\begin{enumerate}[leftmargin=*,label=\rm{(\roman*)}]
\item 
  If  $M \in \hwmod$ has highest weight $n$ and $P:= M^{(n)}$, then there is a morphism of $\D$-modules
  $\varPhi: \verma{P} \to M$ which is the identity on $P$; indeed, $u\cdot M^{(n)} = v \cdot M^{(n)} = 0$.
  Furthermore,  $\hw  (M / \imm \varPhi) < \hw M$.

\medbreak
\item  The Verma modules $\verma{P}$ with $P = P^{(n)} \in \lmod{\D^0}$ are highest weight modules. 
Any highest weight  module $M$ is the quotient of a Verma module, namely of  $\verma{P}$ with $P = M^{(\hw M)}$.

\medbreak
\item The category $\lmod{\D}$ is a subcategory of $\hwmod$.
 Thus any $M\in \lmod{\D}$ has a unique series of submodules 
 $0 =  M_0 \subsetneq M_1 \subsetneq M_2 \dots \subsetneq M_r = M$ 
 (that we call the \emph{hw-series})
 such that $M_i/M_{i-1}$ is a highest weigth module, 
 $i\in \I_{r}$,  and 
 $\hwrank M_i/M_{i-1} > \hwrank M_{i + 1}/M_{i}$,  $i\in \I_{r-1}$.

\medbreak
\item If $M, N \in \hwmod$, then $M \otimes N \in \hwmod$. This follows from Remark \ref{rem:tensor-product}; an
elementary argument shows that the set of pairs $(p,q) \in \varPi(M) \times \varPi(N)$ such that $p+ q = n$ for a given $n$,
is finite when both $\varPi(M)$ and $\varPi(N)$ are bounded above (or below).
\end{enumerate}
 \end{remark}
 
  \subsection{Lowest weight modules}
We also have the  full subcategory $\lwmod$ of $\grLmod{\D}$  consisting of 
 those  $M$ such that $ \varPi(M) $ is bounded below; for $M \in \lwmod$, we set $\lw M \coloneqq \inf \varPi(M)$ and
 $\lwrank M \coloneqq\dim M^{(\lw M)}$.
 If $M \in \hwmod$, then $M^{\dual} \in \lwmod$.  This gives a contravariant equivalence of categories between
 $\hwmod$ and $\lwmod$. 
 
 \medbreak
 Given $P\in \lmod{\D^0}$  such that $P = P^{(n)}$ for some $n\in \Z$,  
 $P$ becomes a $\D^{\leq0}$-module by $x\cdot m = 0$, $y\cdot m = 0$, $m \in P$.
 The  opposite Verma module is
 \begin{align*}
 \vermaopp{P} \coloneqq \Ind_{\D^{\leq0}}^{\D} (P) \simeq \D\ot_{\D^{\leq0}} P.
 \end{align*} 
 Clearly $\vermaopp{P}  \simeq \D^{> 0}\ot_{\ku} P$ as  vector spaces.
Then $M \in \lwmod$ is a \emph{lowest weight}  module if it  is generated by $M^{(\lw M)}$. 
We have the following properties.

\begin{enumerate}[leftmargin=*,label=\rm{(\roman*)}]
\item 
If  $M \in \lwmod$  and $P:= M^{(\lw M)}$, then there is a morphism of $\D$-modules
$\varphi: \vermaopp{P} \to M$, $\varphi_{\vert P} = \id_{P}$.
Moreover,  $\lw  (M / \imm \varphi) > \lw M$.

\smallbreak
\item  The opposite Verma modules  are lowest weight modules and
any lowest weight  module  is the quotient of one of them.

\smallbreak
\item The category $\lmod{\D}$ is a subcategory of $\lwmod$ and any object belonging to  $\lwmod$ and $\hwmod$ is in $\lmod{\D}$. 
Thus any $M\in \lmod{\D}$ has a unique series of submodules 
$0 =  M_0 \subsetneq M_1 \subsetneq M_2 \dots \subsetneq M_t = M$ 
(that we call the \emph{lw-series})
such that $M_i/M_{i-1}$ is a lowest weigth module, 
$i\in \I_{t}$,  and 
$\lwrank M_i/M_{i-1} < \lwrank M_{i + 1}/M_{i}$,  $i\in \I_{t-1}$.

\smallbreak
\item If $M, N \in \lwmod$, then $M \otimes N \in \lwmod$. 

\smallbreak
\item \label{item:dual-not-lw}
 Let $P\in \lmod{\D^0}$ such that $P = P^{(n)}$ for some $n\in \Z$. Then $P^*$ does not generate
 $\verma{P}^{\vee}$.

\smallbreak
\item \label{item:dual-not-verma}  
 Let $P\in \lmod{\D^0}$ such that $P = P^{(n)}$ for some $n\in \Z$. Then
 the  natural map of $\D$-modules $\varphi: \vermaopp{P} \to \verma{P^*}^{\vee}$  is not an isomorphism.
 \end{enumerate}

\noindent\emph{Proof of \ref{item:dual-not-lw}}. Let $p\in P$,  $f\in P^*$ and $a,b \in \ku$. Because of \eqref{eq:def-jordan3}, we have 
\begin{align*}
\langle (au + bv) \cdot f , x\cdot p \rangle &= -\langle  \cdot f , (au + bv)x\cdot p \rangle
\\= - \langle f, &\left(x(au+ b(u+v)) + b(1-g)\right)\cdot p \rangle
= - \langle f,  b(1-g)\cdot p \rangle;
\end{align*}
thus,  $\langle(au + bv) \cdot f  , x\cdot p\rangle = 0 $ if $p\in \ker (1-g)$. Since $\ker (1-g) \neq 0$, we conclude that
$(x\cdot P)^*$ is not contained in the submodule generated by $P^*$.
\qed

\smallbreak
\noindent\emph{Proof of \ref{item:dual-not-verma}}.
The natural  isomorphism  $P\simeq P^{**}$ is of  $\D_0$-modules because   $\Ss_{\D_0}^2 = \id$; thus
 $\varphi$ exists. Now $\imm \varphi$ is generated by $P^{**}$, thus \ref{item:dual-not-lw} applies.
\qed

\section{Modules of highest weight rank one} \label{sec:hwr1}
We aim  to describe indecomposable finite-dimensional highest weight  modules; we start in this section with those having
highest weight 1. We shall define  an indecomposable module $\usrl{n,m} \in \lmod{\D}$ for any $(n,m)\in\N_0^2$.

\medbreak
The only Verma modules of $\hwrank = 1$, with the conventions of this article, are
$\verma{n} \coloneqq  \verma{\mathbf{\lambda}_n}$, $n\in \Z$, 
where $\mathbf{\lambda}_n \in \lmod{\D^{\geq0}}$ has dimension $1$, with basis $\{z_n\}$ and action
\begin{align*}
g\cdot z_n &= z_n, & \zt \cdot z_n &= n z_n, & u\cdot z_n &= 0, & v\cdot z_n &= 0.
\end{align*}
The  elements  $\zn{n}{(i,j)} \coloneqq y^i x^j \cdot z_n$,  $i,j \in\N_0$, form a basis of $\verma{n}$.
Recall that $[t]^{[k]}$ denotes the raising factorial $[t]^{[k]} \coloneqq \prod_{i=1}^{k} (t+i-1)$ for $t\in\ku$ and $k\in\N_0$. 
By \cite[Lemma 2.5]{ap} the action of $\D$  is explicitly given by
\begin{align}\label{eq:action-verma-generators} \allowdisplaybreaks
&\begin{aligned}
x\cdot \zn{n}{(i,j)} &= \sum_{k=0}^i \tbinom{i}{k} \tfrac{k!}{2^k} \zn{n}{(i-k,j+k+1)}, 
\\ y\cdot \zn{n}{(i,j)} &= \zn{n}{(i+1,j)}, \\
g\cdot \zn{n}{(i,j)} &= \sum_{k=0}^i \tbinom{i}{k} \tfrac{[2]^{[k]}}{2^k} \zn{n}{(i-k,j+k)}, 
\\
\zt\cdot \zn{n}{(i,j)} &=  \big(n-2(i+j)\big)  \zn{n}{(i,j)},\\
\end{aligned}
\\ \notag
&\begin{aligned}
u\cdot \zn{n}{(i,j)} &= (-1) \sum_{k=1}^{i-1} \tbinom{i}{k+1} \tfrac{(k+1)!}{2^k} \zn{n}{(i-1-k,j+k)},\\
v\cdot \zn{n}{(i,j)} &= \tfrac{i (n-2j-i+1)}{2} \zn{n}{(i-1, j)} + 
\sum_{k=1}^{i-1}  \tfrac{i!(n-j-i+1)}{(n-k-1)!2^{k+1}} \zn{n}{(i-1-k,j+k)}.
\end{aligned}
\end{align}

It was shown in \cite{adp} that the Verma module
$\verma{n}$  has a unique simple quotient $ \verma{n}/\langle \zn{n}{(0,1)}, \zn{n}{(n+1,0)} \rangle$ which is isomorphic 
to $\smp{n}$, cf. Theorem \ref{th:Ln-simples}. 
Now $\verma{n}$ is presented by the generator  $z_n$ with defining relations
\begin{align}\label{eq:defn-rels-verma}
g\cdot z_n &= z_n,& \zt \cdot z_n &= n z_n, & u\cdot z_n &= 0,  & v\cdot z_n &= 0.
\end{align}
Similarly $\smp{n}$ is presented by  $z_n$ with defining relations \eqref{eq:defn-rels-verma} and
\begin{align}\label{eq:defn-rels-simple}
x \cdot z_n &= 0, &  y^{n+1} \cdot z_n  &= 0.
\end{align}

\begin{definition}
Let $(n,m)\in\N_0^2$. We define $\usrl{n,m} \coloneqq\verma{n+2m}/ \vermasub{n,m}$ where $\vermasub{n,m}$
is the submodule of $\verma{n+2m}$  generated by 
\begin{align}
\label{eq:defn-gens-uniserial--gral}
\begin{aligned}
\zn{n+2m}{(0,m+1)}& = x^{m+1} \cdot z_{n + 2m}, &&
\\
\zn{n+2m}{(n+2(m-j) +1, j)} &= y^{n+2(m-j) +1}x^j \cdot z_n, & j &  \in \I_{0,m}.
\end{aligned}
\end{align}

Thus $\usrl{n,m}$ is presented by  $z = z_{n + 2m}$ with defining  relations 
\begin{align}
\label{eq:defn-rels-vermabis}
&g\cdot z_{n + 2m} = z_{n + 2m},& &\zt \cdot z_{n + 2m} = (n + 2m) z_{n + 2m}, 
\\ \label{eq:defn-rels-uniserial--gral1}
 &u\cdot z_{n + 2m} = 0,  &  & v\cdot z_{n + 2m} = 0,
\\
\label{eq:defn-rels-uniserial--gral}
&x^{m+1} \cdot z_{n + 2m} = 0,&  &y^{n+2(m-j) +1}x^j \cdot z_{n + 2m} = 0, & j &  \in \I_{0,m}.
\end{align}
\end{definition}

Let $\tz_{n,m, (i,j)}$ be the image of $\zn{n+2m}{(i,j)}$ in $\usrl{n,m}$. 
For simplicity, we denote  $\zn{n+2m}{(i,j)}$ by $\znp{i}{j}$ and $\tz_{n,m, (i,j)}$ by $\tzn{i}{j}$.

\begin{lemma}\label{lemma:basis-T-n-m} The set $B = \{\tzn{i}{j} \colon j\in\I_{0,m}, i\in\I_{0,n+2(m-j)}\}$ is a basis of  $\usrl{n,m}$, 
thus
$$\dim \usrl{n,m} = (m+1)(n+m+1).$$
\end{lemma}

\begin{proof}
Let $\mathtt{N}$ be the vector subspace of $\verma{n+2m}$ generated by the  elements
\begin{align}\label{eq:base-N}
&\znp{i}{j}&&\text{with either  } &&\begin{cases} j\geq m+1 &\text{ and } i\in\N_0, \text{ or else} \\
0\leq j \leq m &\text{ and } i\geq n+2(m-j) +1.
\end{cases}
\end{align}
We claim that $\vermasub{n,m} = \mathtt{N}$; clearly this equality implies the Lemma. 
For this claim, since the generators  \eqref{eq:defn-gens-uniserial--gral} belong to $\mathtt{N}$,
we are reduced to prove
\begin{multicols}{2}
\begin{enumerate}[leftmargin=*,label=\rm{(\roman*)}]
\item \label{item:submod2}$\vermasub{n,m} \supset \mathtt{N}$.
\item\label{item:submod1} $\mathtt{N}$ is a submodule of  $\verma{n+2m}$, 

\end{enumerate}
\end{multicols}

\ref{item:submod2} is clear: if $j\geq m+1$, then $\znp{i}{j} = y^ix^{j-m-1} \cdot \znp{0}{m+1}$; if
$j \in \I_{0, m}$ and $i\geq n+2(m-j) +1$, then $\znp{i}{j} = y^{i-n-2(m-j) -1}  \cdot \znp{n+2(m-j) +1}{j}$.
 
 \medbreak
\ref{item:submod1}:
We use the formulas \eqref{eq:action-verma-generators} to show that the generators leave $\mathtt{N}$ invariant. 
Fix $\znp{i_0}{j_0}$ as in \eqref{eq:base-N}. If $j_0\geq m+1$, then by 
\eqref{eq:action-verma-generators} the actions of the generators involve linear combinations of $\znp{i}{j}$ with $j\geq j_0$ hence they are in $\mathtt{N}$. 
So we assume that $0\leq j_0 \leq m$ and $i_0\geq n+2(m-j_0) +1$.

\subsection*{Action of $x$ on $\znp{i_0}{j_0}$} Here $x \cdot \znp{i_0}{j_0}$ is a linear combination of elements of the form $\znp{i_0-k}{j_0+k+1}$ with $k\in \I_{0, i_0}$. If $j_0 + k + 1\geq m+1$, then  we are done. Otherwise,
\begin{align*}
i_0-k &\geq n+2m - 2j_0 + 1 - k \geq   n +2m - 2(j_0+k+1) + 1.
\end{align*} 

\subsection*{Action of $y$ or $\zt$ on $\znp{i_0}{j_0}$} This is clear.

\subsection*{Action of $g$ on $\znp{i_0}{j_0}$} Here $g \cdot \znp{i_0}{j_0}$ is a linear combination of elements of the form $\znp{i_0-k}{j_0+k}$ with $k\in \I_{0, i_0}$.  If $j_0 + k \geq m+1$, then we are done. Otherwise,
\begin{align*}
i_0-k &\geq n+2m - 2j_0 + 1 - k \geq  n +2m - 2(j_0+k) + 1.
\end{align*}

\subsection*{Action of $u$ on $\znp{i_0}{j_0}$} Here $u \cdot \znp{i_0}{j_0}$ is a linear combination of elements of the form 
$\znp {i_0-k -1} {j_0+k}$ with $k\in \I_{i_0-1}$.  If $j_0 + k \geq m+1$, then we are done. Otherwise,
\begin{align*}
i_0-k-1 &\geq n + 2m - 2 j_0 +1 -k -1 \geq  n + 2m - 2(j_0+k) + 1. 
\end{align*}

\subsection*{Action of $v$ on $\znp{i_0}{j_0}$} 
Here $v \cdot \znp{i_0}{j_0}$ is a linear combination of 
$\znp {i_0-1}{j_0})$ and elements of the form $\znp{i_0-k-1} {j_0+k}$ with $k\in \I_{i_0 - 1}$. 
 We begin with the latter case. If $j_0+k \geq m+1$, then we are done. Otherwise,
\begin{align*}
i_0-1-k &\geq n + 2m - 2 j_0 + 1 -1 -k \geq n + 2m - 2(j_0+k)+1. 
\end{align*}
Now for $z(i_0-1,j_0)$, if $i_0 = n+2(m-j_0) +1$ then the coefficient that goes with $z(i_0-1,j_0)$ is zero. So we can assume $i_0 > n+2(m-j_0) +1$. Then $i_0 - 1 \geq n+2(m-j_0) +1$ and we are done.
\end{proof}

Clearly, $\usrl{n,0} \simeq \smp{n}$.  

\begin{lemma}\label{lemma:Ln-is-sub-module-of-Tnm}
The linear map $\psi: \smp{n}\hookrightarrow \usrl{n,m}$ given by 
$\zn{n}{i}\mapsto \tzn{i}{m}$ for $i\in\I_{0,n}$, is a  monomorphism of $\D$-modules.
 Let $\mathtt{L} \coloneqq \imm \psi$.   If $m \geq 1$, then 
\[ \usrl{n,m} / \mathtt{L} \simeq \usrl{n+2,m-1}.\]
\end{lemma}

\pf
Using the formulas \eqref{eq:action-verma-generators} we see first that $\tzn{0}{m}$ satisfies the defining relations
\eqref{eq:defn-rels-verma} and \eqref{eq:defn-rels-simple} of $\smp{n}$, implying the existence of $\psi$
with the desired properties. Second, we see that the class of $\tzn{0}{0}$ in $\usrl{n,m} / \mathtt{L}$ 
satisfies the defining relations
\eqref{eq:defn-rels-vermabis}, \eqref{eq:defn-rels-uniserial--gral1} and \eqref{eq:defn-rels-uniserial--gral} of $\usrl{n+2,m-1}$.
Thus we have an epimorphism $\usrl{n+2,m-1} \twoheadrightarrow\usrl{n,m} / \mathtt{L}$ which is an isomorphism
because $\dim\usrl{n+2,m-1} = m(n+m+2) = \dim\usrl{n,m} / \mathtt{L}$.
\epf

\begin{lemma}\label{lemma:Ln-is-contained-in-every-sub-module-of-Tnm}
Let $N$ be a non-zero submodule of $\usrl{n,m}$. Then $\mathtt{L} \subseteq N$.
\end{lemma}
\begin{proof}
We first show that for every $j\in\I_{m}$ and $i\in\I_{0,n+2(m-j)}$ we have
\begin{align}\label{eq:action-vi}
v^{i}\cdot\tzn{i}{j} = \frac{(i!)^2}{2^i} \binom{n+2(m-j)}{i}\tzn{0}{j}.
\end{align}
We argue recursively on $i$. For $i=0$ the equality is clear. Suppose that
\begin{align*}
v^{\ell}\cdot\tzn{\ell}{j} = \frac{(\ell!)^2}{2^\ell} \binom{n+2(m-j)}{\ell}\tzn{0}{j}, \qquad \text{ for } \ell\in\I_{0,i}.
\end{align*}
Using \eqref{eq:action-verma-generators} we see that  $v^{i+1}\cdot\tzn{i+1}{j}$ is equal to
\begin{align*}
&\tfrac{(i+1)(n+2(m-j)-i)}{2}  v^i \tzn{i}{j} +\sum_{k=1}^i \tbinom{i+1}{k+1} \tfrac{(n+2m-j-i)(k+1)!}{2^{k+1}}
 v^{k} v^{i-k}\cdot \tzn{i-k}{j+k}\\
&=\tfrac{(i+1)(n+2(m-j)-i)(i!)^2}{2^{i+1}} \tbinom{n+2(m-j)}{i}\tzn{0}{j}\\
&+\sum_{k=1}^i \tbinom{i+1}{k+1} \tfrac{(n+2m-j-i)(k+1)! ((i-k)!)^2}{2^{i+1}}\tbinom{n+2(m-j-k)}{i-k} v^k\tzn{0}{j+k}\\
&=\tfrac{((i+1)!)^2}{2^{i+1}} \tbinom{n+2(m-j)}{i+1}\tzn{0}{j}.
\end{align*}
Clearly \eqref{eq:action-vi} implies that $v^k \cdot \tzn{i}{j} = 0$ for $k > i$.
Now let $z\in N - 0$. Then $z = \sum_{i,j} c_{i,j} \tzn{i}{j}$ for some  $c_{i,j}\in \ku$. Let $i_0 = \max\{i \colon c_{i,j}\neq 0\}$. Then
\begin{align*}
v^{i_0}\cdot z = \sum_j c_{i_0,j}\frac{(i_0!)^2}{2^{i_0}} \binom{n+2(m-j)}{i_0}\tzn{0}{j}\in N.
\end{align*}
Taking $j_0 = \max\{j\colon c_{i_0,j}\neq 0\}$ we get that
\begin{align*}
x^{m-j_0} v^{i_0}\cdot z &= c_{i_0,j_0}\frac{(i_0!)^2}{2^{i_0}} \binom{n+2(m-j_0)}{i_0}\tzn{0}{m}\in N.
\end{align*}
Hence $\tzn{0}{m} \in N$ and since $\mathtt{L}$ is simple, this shows that $\mathtt{L}\subseteq N$.
\end{proof}
\begin{prop}
The module $\usrl{n,m}$ is uniserial and indecomposable.
\end{prop}
\begin{proof}
If $N$ is a simple submodule of $\usrl{n,m}$, then $N = \mathtt{L}$ by Lemma \ref{lemma:Ln-is-contained-in-every-sub-module-of-Tnm}, thus   $\mathtt{L}$ is the socle of $\usrl{n,m}$. 
We conclude from Lemma \ref{lemma:Ln-is-sub-module-of-Tnm}  that the socle series is a composition series, hence 
$\usrl{n,m}$ is uniserial and  indecomposable. 
\end{proof}

\begin{remark}
The dual module $\usrl{n,m}^*$ is also uniserial and indecomposable but it is not a highest weight module;
the subfactors of its hw-series are the simple modules $\smp{n}$ etc.
\end{remark}

\begin{theorem}\label{thm:indecomposables-rank-1-class}
Let $\mathtt{T}$ be an indecomposable finite-dimensional highest weight module with $\hwrank \mathtt{T} = 1$. Then $\mathtt{T} \simeq \usrl{n,m}$
for some $(n,m)\in \N_{0}$.
\end{theorem}
\pf
 Assume that $\hw \mathtt{T} = p \in \Z$. Since   $\mathtt{T}$ 
 is generated by $\mathtt{T}^{(p)}$,  $\mathtt{T}$ is a quotient of $\verma{p}$. 
Fix $\mathtt{z} \in \mathtt{T}^{(p)} - 0$. Since $\hwrank \mathtt{T} = 1$, $\mathtt{z}$ 
generates $\mathtt{T}$, $\zt \cdot \mathtt{z} = p \mathtt{z}$ and $ g \cdot \mathtt{z} =  \mathtt{z}$. 
Given a simple quotient $L$ of $\mathtt{T}$, $L$ is then a (finite-dimensional) simple quotient of $\verma{p}$. Then $p\in \N_{0}$ and $L \simeq L(p)$  as in \cite[3.11]{adp}. 

\smallbreak
 By  Proposition \ref{prop:g-1-x-u-D-nilpotent},  $x^m \cdot \mathtt{z} \neq 0$ and $x^{m+1} \cdot \mathtt{z} = 0$
 for some $m\in \N_0$. 
Let $n \coloneqq p - 2m$. Applying the relations \eqref{eq:action-verma-generators}, we see that 
the submodule generated by $x^m \cdot \mathtt{z}$ is a quotient of $M(n)$; hence $n \in \N_0$ arguing as above.

\smallbreak
Let $j\in \I_{0, m}$. By Remark \ref{rem:locally-nilpotent} there exists $a_j\in \N_0$ such that
$y^{a_j}x^j \cdot \mathtt{z} \neq 0$, $y^{a_j + 1}x^j \cdot \mathtt{z} = 0$. 
Then  the last commutation relation in \eqref{eq:action-verma-generators} says that
\begin{align}\label{eq:proof-T-rank-1}
\begin{aligned}
0 = v y^{a_j + 1}x^j \cdot \mathtt{z} &= \tfrac{(a_j + 1) (p-2j-a_j)}{2} y^{a_j}x^j \cdot \mathtt{z} 
\\ & \qquad + 
\sum_{k=1}^{a_j}  \tfrac{(a_j + 1)!(p-j-a_j)}{(p-k-1)!2^{k+1}} y^{a_j-k}x^{j+k} \cdot \mathtt{z}.
\end{aligned}\end{align}
Let $I = \{k\in \I_{0,a_j} :  y^{a_j-k}x^{j+k} \cdot \mathtt{z} \neq 0\}$. Then $\{ y^{a_j-k}x^{j+k} \cdot \mathtt{z} : k\in I\}$ is linearly independent. Indeed if $\sum_{k\in I} c_k  y^{a_j-k}x^{j+k} \cdot \mathtt{z} = 0$, then applying $y$ enough times one can show that each $c_k$ should be zero. Then equation \eqref{eq:proof-T-rank-1} tell us that $p-2j-a_j = 0$ since $0\in I$. Then $a_j = p - 2j = n - 2(m-j)$ by definition of $n$. Hence $\mathtt{T}$ is a quotient of $\mathtt{T}(n,m)$. Let $\mathtt{z}_{i,j} \coloneqq y^i x^ j \cdot \mathtt{z}$. By an argument similar to the one given in Lemma \ref{lemma:basis-T-n-m}, one showa that the elements $\{\mathtt{z}_{i,j}\}$ form a basis of $\mathtt{T}$, hence $\mathtt{T}\simeq \mathtt{T}(n,m)$.
\epf

\begin{remark}
The previous result shows that the requirement in Definition \ref{def:hwmodule} that  highest weight modules 
are generated by their highest weight subspaces is necessary. 
Indeed, otherwise there would be many more 
indecomposables than in Proposition \ref{thm:indecomposables-rank-1-class}. For instance
let $\widetilde{\mathtt{T}} = \mathtt{T}(n,m) \oplus S_\gamma(n)$, 
where $m\geq 1$ and $S_\gamma(n)$ is the indecomposable module defined in 
Proposition \ref{prop:Sn},  let $N$ be the submodule generated by the 
elements $s_i - \mathtt{z}_{i,m}$, and let $\mathtt{T} = \widetilde{\mathtt{T}}/N$. 
Then $\mathtt{T}$ is easily seen to be indecomposable since it satisfies a 
property similar to the one proven in Lemma \ref{lemma:Ln-is-contained-in-every-sub-module-of-Tnm}. 
It satisfy $\hwrank \mathtt{T} = 1$, but it is not isomorphic to any $\mathtt{T}(n',m')$ 
since it has two copies of $\mathtt{L}(n)$ as composition factors.
\end{remark}

\section{Modules of highest weight rank two} \label{sec:hwr2}

In this Section we introduce a family $\usrldos{n}{\gamma}$ of highest weight modules of  $\hwrank 2$ and use it to classify self extensions of simple modules. 

\subsection{Highest weight rank 2}
Let $n\in \N_0$ and $(\lambda, \mu)\in \ku^2$.
We consider  the $\D^0$-module $\indescdos{n}{\lambda}{\mu}$ with basis $\stt, \wtt$ and action
\begin{align}\label{eq:vermados-1}
g\cdot \stt &= \stt,& g\cdot \wtt &= \wtt + \lambda \stt,
& \zt \cdot \stt &=n \stt,  &  \zt \cdot \wtt &=  n\wtt + \mu \stt.
\end{align}
The  module $\indescdos{n}{\lambda}{\mu}$ is isomorphic to $\indescdos{n}{t\lambda}{t\mu}$ for any $t\in \kut$ and is indecomposable whenever $(\lambda, \mu)  \neq 0$; any indecomposable in $\lmod{\D^0}$ of dimension 2 has this shape.
Let $\vermados{n}{\lambda}{\mu} = \Ind_{\D^{\geq0}}^{\D} (\indescdos{n}{\lambda}{\mu})$; it  is presented by generators $\stt$ and $\stt$ with defining relations \eqref{eq:vermados-1} and
 \begin{align}\label{eq:vermados-2}
u\cdot \stt  &= 0,  & v\cdot \stt &= 0, & u\cdot \wtt  &= 0,  & v\cdot \wtt &= 0.
\end{align}
Here is a basis of $\vermados{n}{\lambda}{\mu}$:
\begin{align*}
\stt_{i,j} &\coloneqq y^i x^j \cdot \stt,& \wtt_{i,j} &\coloneqq y^i x^j \cdot \wtt,&  i,j &\in\N_0.
\end{align*}
Let $\Mtt_1$ be the span of $(\stt_{i,j})_{i,j \in\N_0}$ and $\Mtt_2 = \vermados{n}{\lambda}{\mu} / \Mtt_1$.
Thus we have a short exact sequence $\xymatrix{ \verma{n}\simeq \Mtt_1 \ar@{^{(}->}[r]   & \vermados{n}{\lambda}{\mu}  \ar@{->>}[r] & \Mtt_2 \simeq\verma{n} }$
in $\Lmod{\D}$. 

\medbreak
In the rest of this Subsection,
$N$ is a submodule of $\vermados{n}{\lambda}{\mu}$, 
\begin{align*}
S &\coloneqq \vermados{n}{\lambda}{\mu}/ N, & S_1 &\coloneqq \Mtt_1 / \Mtt_1 \cap N \hookrightarrow S,&
S_2 &\coloneqq S/S_1 \simeq \Mtt_2/ \left(N /\Mtt_1 \cap N\right).
\end{align*}
Let $s,w,s_{ij}, w_{ij}$ be the images of  $\stt,\wtt,\stt_{ij}, \wtt_{ij}$ in $S$;  let $\overline{r} \in S_2$ 
be the image of $r \in S$ under the canonical projection.

\subsubsection{$S_1$ and $S_2$ are simple}
\begin{lemma}\label{lema:-selfdual} Keep the notation above.

\begin{enumerate}[leftmargin=*,label=\rm{(\roman*)}]
\item\label{item:lema-selfdual-1}
Assume that the following relations hold in $S$ and $S_2$:
\begin{align}
\label{eq:defn-rels-uniserial-dos}
x \cdot s &= 0,
\\\label{eq:defn-rels-uniserial-tres}
x \cdot \overline{w} &= 0.
\end{align} 
Then $S_1$ is spanned by $s_i \coloneqq s_{i,0} = y^{i}\cdot s$, $i \in \N_0$;
and $S_2$ is spanned by $\overline{w_i}$ where   $w_i \coloneqq w _{i,0} = y^{i}\cdot w$, $i \in \N_0$.
Hence
\begin{align}\label{eq:defn-rels-uniserial-directsum}
S^{(n - 2i)} &= \ku  s_i  + \ku  w_i,  \quad i \in \N_{0}, &&\text{and} & S &= \oplus_{i  \in \N_{0}} S^{(n - 2i)} .
\end{align}
In addition, there exists $\gamma \in \ku$ such that $x \cdot w  = \gamma s_1$.

\item\label{item:lema-selfdual-2}  The following relations hold for all $i \in \N_0$:
\begin{align}
\label{eq:defn-rels-uniserial-x-si}
x\cdot s_i &= 0, \quad u \cdot s_{i} = 0, \quad v \cdot s_{i} = \tfrac{i (n-i+1)}{2} s_{i-1},
\\
\label{eq:defn-rels-uniserial-x-wi}
x \cdot w_i  &= \gamma s_{i+1}, 
\\ \label{eq:defn-rels-uniserial-z-wi}
\zt \cdot w_i &= (n-2i)  w_i + \mu  s_i,
\\ 
\label{eq:defn-rels-uniserial-g-wi}
g \cdot w_i  &= w_i + (\lambda + i\gamma)s_i, 
\\ \label{eq:defn-rels-uniserial-u-wi}
u \cdot w_i &=  - \left(i\lambda + \tfrac{i(i-1)}{2}\gamma \right) \,  s_{i-1},
\\\label{eq:defn-rels-uniserial-v-wi}
v\cdot w_i &= \tfrac{i(n-i+1)}{2} w_{i-1}  
+\left(\tfrac{i(n-2i+2)}{2} \lambda + \tfrac{i(i-1)(n -2i+2)}{4} \gamma - i \mu \right) s_{i-1}.
\end{align}
\end{enumerate}
\end{lemma}

\pf \ref{item:lema-selfdual-1} follows directly from \eqref{eq:defn-rels-uniserial-dos} 
and \eqref{eq:defn-rels-uniserial-tres}, looking at the basis of $\vermados{n}{\lambda}{\mu}$.
Then \eqref{eq:weight-action} implies \eqref{eq:defn-rels-uniserial-directsum}. Now  
$x \cdot w \in S_1^{(n-2)}$ hence  there exists $\gamma \in \ku$ such that $x \cdot w  = \gamma s_1$ by \eqref{eq:defn-rels-uniserial-tres}.  The relations in \ref{item:lema-selfdual-2} are proved arguing  recursively. 
For\eqref{eq:defn-rels-uniserial-x-si} the defining relations of $\D$ are used.
For  \eqref{eq:defn-rels-uniserial-x-wi} we have 
\begin{align*}
x \cdot w_{i + 1} =  xy \cdot w_i  = (y+ \tfrac{1}{2} x)x \cdot w_i  = 
\gamma (y+ \tfrac{1}{2} x)  \cdot s_{i+1} \overset{\eqref{eq:defn-rels-uniserial-x-si}}{=} \gamma s_{i+2}.
\end{align*}
The proof of \eqref{eq:defn-rels-uniserial-z-wi} is direct starting from \eqref{eq:vermados-1}:
\begin{align*}
\zt \cdot w_{i+1} = \zt y\cdot w_{i}  = (y \zt  -2 y) \cdot w_{i} = 
(n-2(i + 1)) w_{i+1} + \mu  s_{i+1}.
\end{align*}
The proof of \eqref{eq:defn-rels-uniserial-g-wi} also starts from \eqref{eq:vermados-1}:
\begin{align*}
g\cdot w_{i + 1} &= gy \cdot w_i  = (y+  x)g\cdot w_i = (y+  x) \cdot (w_i + (\lambda + i\gamma) s_i) 
\\
& \overset{\eqref{eq:defn-rels-uniserial-x-si}}{=}  w_{i +1}+ (\lambda + i\gamma) s_{i +1} + \gamma s_{i+1}  =  w_{i +1}+ (\lambda + (i+1)\gamma) s_{i +1}.
\end{align*}
To prove  \eqref{eq:defn-rels-uniserial-u-wi} we start from \eqref{eq:vermados-2} and argue:
\begin{align*} 
u\cdot w_{i+1} &= u y\cdot w_{i} = (yu  +1 - g ) \cdot w_{i} 
\overset{\eqref{eq:defn-rels-uniserial-g-wi}}{=} \left(-(i\lambda + \tfrac{i(i-1)}{2}\gamma)   - (\lambda + i\gamma) \right) s_i
\\&= -\left((i+1)\lambda x+ \tfrac{i(i+1)}{2}\gamma\right) \, s_{i}
\end{align*}
as needed.
Finally we prove \eqref{eq:defn-rels-uniserial-v-wi}. First we observe that
\begin{align*}
-g\zeta\cdot w_i &= g \cdot (\tfrac{(n-2i)}{2} w_i - \mu  s_i) =  \tfrac{(n-2i)}{2}w_i + (\tfrac{(n-2i)}{2}\lambda +\tfrac{(n-2i)i}{2} \gamma - \mu )s_i.
\end{align*}
Then we proceed recursively:
\begin{align*}
v\cdot w_{i+1} =& v y\cdot w_{i} = (y v -g \zeta  + y u)  \cdot w_{i} 
\\
=& \tfrac{i(n-i+1)}{2} w_{i}  
+\left(\tfrac{i(n-2i+2)}{2} \lambda + \tfrac{i(i-1)(n -2i+2)}{4} \gamma - i \mu \right) s_{i}
\\
& + \tfrac{(n-2i)}{2}w_i + (\tfrac{(n-2i)}{2}\lambda +\tfrac{(n-2i)i}{2} \gamma - \mu )s_i - \left(i\lambda + \tfrac{i(i-1)}{2}\gamma \right) \,  s_{i};
\end{align*}
it remains to perform the routine verification of  the following equalities:
\begin{align*}
\tfrac{i(n-i+1)}{2} + \tfrac{n-2i}{2} &= \tfrac{(i + 1) \left(n-(i + 1)+1\right)}{2},
\\
\tfrac{i(n-2i+2)}{2} + \tfrac{n-2i}{2} -i &=  \tfrac{(i + 1) \left(n- 2(i + 1) + 2\right)}{2},
\\
\tfrac{i(i-1)(n -2i+2)}{4} +\tfrac{(n-2i)i}{2} - \tfrac{i(i-1)}{2} &= \tfrac{i(i+1) \left(n -2(i+1)+2\right)}{4},
\end{align*}
a task left to the readers. \epf

Fix $\gamma \in \ku$ and let  $N$ be the submodule of 
$\vermados{n}{\lambda}{\mu}$ generated by 
$x \cdot \stt$, $y^{n+1} \cdot \stt$, $x \cdot \wtt  - \gamma y \cdot \stt$  $y^{n+1} \cdot \wtt$. 
That is, the module $S =\vermados{n}{\lambda}{\mu}/N$ 
is presented by generators $\stt$ and $\wtt$ with defining relations
\eqref{eq:vermados-1}, \eqref{eq:vermados-2}, 
\begin{align}
\label{eq:defn-rels-uniserial-s}
x \cdot \stt &= 0, &y^{n+1} \cdot \stt  &= 0, 
\\\label{eq:defn-rels-uniserial-w}
x \cdot \wtt &= \gamma y \cdot \stt, & y^{n+1} \cdot \wtt  &= 0.
\end{align}  

\begin{prop}\label{prop:Sn} Let $S$ be as above.
\begin{enumerate}[leftmargin=*,label=\rm{(\roman*)}]
\item\label{item:Sn1}  The module $S$ has dimension $n^2$ if an only if
\begin{align}\label{eq:conditions-Sn-gamma}
\lambda + \tfrac{n}{2}\gamma  &=0 & &\text{ and }&
\mu &=0.
\end{align}

\item\label{item:Sn2} Set $\usrldos{n}{\gamma} \coloneqq S$ when $(\lambda, \mu) = (-\tfrac{n}{2}\gamma, 0)$; then 
$\usrldos{n}{\gamma}  \simeq \usrldos{n}{\gamma'}$ if and only if $\gamma = t\gamma'$ for some $t\in \kut$.

\item\label{item:Sn3} For any $\gamma$, $\usrldos{n}{\gamma}$  is an extension of $L_n$ by $L_n$. 
Any extension is like this and $\dim \Ext^1_{\D}(\smp{n},\smp{n}) = 1$, $n\in \N_0$.
\end{enumerate}
\end{prop}

For our conventions on extensions, see Subsection \ref{subsec:extensions}.

\pf We keep the notation above and apply Lemma \ref{lema:-selfdual}; thus $\dim S \leq n^2$. 
By \eqref{eq:defn-rels-uniserial-u-wi} and \eqref{eq:defn-rels-uniserial-v-wi}, we see  that
\begin{align}\label{eq:defn-rels-uniserial-u-wi2}
u \cdot w_{n+1} &=  - (n+1)\left(\lambda + \tfrac{n}{2}\gamma \right) \,  s_{n},
\\ \label{eq:defn-rels-uniserial-v-wi2}
v\cdot w_{n+1}&=  
-(n+1)\left(\tfrac{n}{2} \lambda + \tfrac{n^2}{4} \gamma + \mu \right) s_{n}.
\end{align}
If $\lambda + \tfrac{n}{2}\gamma \neq  0$, then \eqref{eq:defn-rels-uniserial-u-wi2} says that $\stt_{n} \in N$.
If $\lambda + \tfrac{n}{2}\gamma = 0$ but $\mu \neq 0$, then \eqref{eq:defn-rels-uniserial-v-wi2} says that $\stt_{n} \in N$.
In both cases,  $\stt \in N$ by \eqref{eq:defn-rels-uniserial-x-si}, hence $\dim S \leq n$. 
Conversely assume that \eqref{eq:conditions-Sn-gamma} holds. Then \eqref{eq:weight-action} implies that
\begin{align*}
N \subset \oplus_{k\in \N} \vermados{n}{\lambda}{\mu}^{(n - 2k)}.
\end{align*}
Then $N \cap  \vermados{n}{\lambda}{\mu}^{(n)}$, in particular $\stt \notin N$.  
Now there are morphisms of $\D$-modules $\iota:\smp{n} \to S_1$
and $\widetilde{\iota}:\smp{n} \to S_2$
given by $\iota (\zn{n}{i}) = s_i$ and $\widetilde{\iota} (\zn{n}{i}) = \overline{w}_i$, $i \in \I_{0,n}$.
Now $s\neq 0$ in $S$ 
thus $\iota$ is injective (because $\smp{n} $ is simple) and $\{s_0,\dots,s_n\}$ are linearly independent.
Similarly $\overline{w} \neq 0$ in $S_2$,  
 $\widetilde{\iota}$ is injective and $\{w_0,\dots,w_n\}$ are linearly independent.
 Hence $\dim S = n^2$ and \ref{item:Sn1} is proved.

\ref{item:Sn2}: The $\D^0$-module $\usrldos{n}{0}^{(n)} \simeq \indescdos{n}{0}{0}$ is decomposable
while $\usrldos{n}{\gamma}^{(n)} \simeq \indescdos{n}{\gamma}{0}$ is indecomposable  when $\gamma \neq 0$.

$\pi:\usrldos{n}{\gamma} \to \smp{n}$ given by
$\pi(s_i) =0$ and $\pi(w_i) = \zn{n}{i}$, $i \in \I_{0,n}$. Now $s\neq 0$ in $\usrldos{n}{\gamma}$,
thus $\iota$ is injective (because $\smp{n} $ is simple) and $\{s_0,\dots,s_n\}$ are linearly independent;
while $\{w_0,\dots,w_n\}$ are linearly independent because they project to a basis of $\smp{n}$.

Indeed let $\iota\colon L_n\longrightarrow \usrldos{n}{\gamma} $ given by $s_i\mapsto s_i$ and 
$\pi\colon \usrldos{n}{\gamma}  \longrightarrow L_n$ given by $s_i\mapsto 0$, $w_i \mapsto s_i$. Then the following sequence is
is exact
\begin{align*}
\xymatrix{ L_n  \ar@{^{(}->}[rr] ^{\iota} & & \usrldos{n}{\gamma}  \ar@{->>}[rr] ^{\pi} & & L_{n} },
\end{align*}

\epf

\begin{remark}
In the context of Lemma \ref{lema:-selfdual}, 
if $\dim S < \infty$ and $x \cdot s = 0$, then $S_1$ actually  belongs to $\lmod{U(\spl_2)}$, thus $y^{n+1} \cdot s  = 0$;
similarly $\overline{w} = 0$ implies that $y^{n+1} \cdot \overline{w}  = 0$, hence $y^{n+1} \cdot w  = 0$ by looking at the weights of $S_1$.
\end{remark}

\begin{remark}
The set   $\{s_0,\dots,s_n,w_0,\dots,w_n\}$ is a basis of $\usrldos{n}{\gamma}$.
Set $w_{n+1} = s_{n+1} = w_{-1} = s_{-1} = 0$ by convention.  Then 
the action is given by
\begin{align}\label{eq:Sn-D-action}
\begin{aligned}
x \cdot s_{i} &= 0, &   x\cdot w_i &= \gamma s_{i+1}, 
\\  
y \cdot s_{i} &= s_{i+1}, & y\cdot w_i &= w_{i+1}, 
\\
g\cdot s_{i} &= s_{i}, & g\cdot w_i &= w_i -  \tfrac{n - 2i}{2}\gamma \,s_i,  
\\
\zt\cdot s_{i}&= (n-2i) s_{i},  & 
\zt\cdot w_i &= (n-2i) \,w_i, 
\\  
u \cdot s_{i} &= 0, & u \cdot w_i &=   \tfrac{i(n+1-i)}{2} \gamma \,  s_{i-1}, 
\\
v \cdot s_{i} &= \tfrac{i (n-i+1)}{2} s_{i-1}, & 
v\cdot w_i &= \tfrac{i(n-i+1)}{2} w_{i-1}  
- \tfrac{i(n -2i+2) (n+1-i)}{4} \gamma  \, s_{i-1}.
\end{aligned}
\end{align}

\end{remark}

\section{Extensions of simple modules}\label{sec:extensions-simple-modules}
\subsection{Generalities}\label{subsec:extensions}
Let $n,m \in \N_0$. The goal of  this Section is to compute the vector spaces
$\Ext^1_{\D}(\smp{n},\smp{m})$. 
To fix the notation,   an extension of $\smp{n}$ by $\smp{m}$ 
is a short exact sequence
\begin{align}\label{eq:extension-Ln-Lm}
\xymatrix{ \smp{n}  \ar@{^{(}->}[rr] ^{\iota} & & T \ar@{->>}[rr] ^{\pi} & & \smp{m} .}
\end{align}

By abuse of notation we say also that $T$ is an extension   of $\smp{n}$ by $\smp{m}$. 

\begin{remark}\label{rem:extensions-dual}
Let $H$ be a Hopf algebra and $M, N, P \in \lmod{H}$.
Given  an extension $\xymatrix{M  \ar@{^{(}->}[r] ^{\iota} &  P \ar@{->>}[r] ^{\pi} &  N,}$ the sequence    
$\xymatrix{ N^*  \ar@{^{(}->}[r] ^{\pi^*} &  P^* \ar@{->>}[r] ^{\iota^*} &  M^* }$
is  exact.
Then $\Ext^1_{H}(M, N) \simeq \Ext^1_{H}( N^*, M^*)$. Since $\smp{p}^* \simeq \smp{p}$, $p\in \N_0$, we have
\begin{align*}
\Ext^1_{\D}(\smp{n}, \smp{m}) &\simeq \Ext^1_{\D}( \smp{m}, \smp{n}), & m,n&\in \N_0.
\end{align*}

\end{remark}

Fix an extension \eqref{eq:extension-Ln-Lm}
and  identify $\smp{n}$ as a submodule of $T$ via $\iota$. 
Pick $w_m(0)\in T$ such that $\pi(w_m(0)) = z_m(0)\in \smp{m}$ and set 
\begin{align*}
w_m(i) &\coloneqq y^i \cdot w_m(0), & i &\in\I_{0,m}.
\end{align*}
Then $\{\zn{n}{0},\dots,\zn{n}{n},w_m(0),\dots,w_m(m)\}$ is a linear basis of $T$.
Also set
\begin{align*}
r_d &\coloneqq d \cdot w_m(0), & d &\in \Nuc = \ku\langle x,u,g^{\pm 1}\rangle. 
\end{align*}
Now $\Nuc$ acts by 0 on $\smp{m}$, hence $r_d \in \smp{n}$ for  $d\in \Nuc$.
Also $v \cdot w_m(0) \in \smp{n}$.

We start  giving restrictions on the existence of non-trivial extensions.
\begin{lemma}\label{lema:non-trivial-extensions-Ln-Lm} If  $\Ext^1_{\D}(\smp{n},\smp{m}) \neq 0$,
then $m-n\in \{2, 0, -2\}$.
\end{lemma}
\noindent\emph{Proof.} \ 
Let $T$ be an extension   of $\smp{n}$ by $\smp{m}$ and keep the notation above.

\begin{claim} If $r_x = r_{g-1} = r_u = 0$, then the extension $T$ is trivial. 
\end{claim}

We claim that $x$, $g-1$ and $u$ act by $0$ on $T$. Since they
act trivially on $\smp{n}$, it is enough to consider the  action on the $w_m(i)$'s. 
We argue recursively, the case $i=0$ being the hypothesis. 
If $x$, $g-1$ and $u$ act by $0$ in $w_m(i)$, then
\begin{align*}
&x \cdot w_m(i+1) = xy \cdot w_m(i) = (yx + \frac{1}{2}x^2)w_m(i) = 0,\\
&u \cdot w_m(i+1) = uy \cdot w_m(i) = (yu + (1-g))\cdot w_m(i) = 0,\\
&(1-g) \cdot w_m(i+1) = (1-g)y \cdot w_m(i) = (y(1-g) - gx) \cdot w_m(i) = 0.
\end{align*}
Hence $T\in \lmod{\D/\D\Nuc^+}$; since $\D/\D\Nuc^+ \simeq U(\spl_2(\ku))$, the extension is trivial.

\medbreak
Thus, if $T$ is a non-trivial extension, then at least one of  $r_x$, $r_{g-1}$, $r_u$ is not zero.
Let $s\in \smp{n}$ be such that $\zeta\cdot w_m(0) = m w_m(0) + s$.

\begin{case}
$r_u\neq 0$. Since
\begin{align*}
v\cdot r_u = vu \cdot w_m(0) = (uv - \tfrac{1}{2} u^2)\cdot w_m(0) = 0,
\end{align*}
there exists $a\in \kut$ such that $r_u = a \zn{n}{0}$. We compute in two ways:
\begin{align*}
\zt \cdot r_u &= \zt u \cdot w_m(0) = (u\zt + 2u)\cdot w_m(0) = m r_u + u\cdot s + 2 r_u = (m+2) r_u
\end{align*}
and also $\zt \cdot r_u = a \zt\cdot \zn{n}{0} = n r_u$.
We conclude that $n = m+2$.
\end{case}   

\begin{case} $r_u = 0$ and $r_{g-1} \neq 0$.  Since
\begin{align*}
v\cdot r_{g-1} = v(g-1) \cdot w_m(0) = ((g-1)v + gu)\cdot w_m(0) = g\cdot r_u = 0,
\end{align*}
there exists $a\in \kut$ such that $r_{g-1} = a \zn{n}{0}$. We compute in two ways:
\begin{align*}
\zt \cdot r_{g-1} &= (g-1)\zt\cdot w_m(0) = m r_{g-1} + (g-1)\cdot s = m r_{g-1}
\end{align*}
and also $\zeta \cdot r_{g-1} = a \zt\cdot \zn{n}{0} = n r_{g-1}$. We conclude that $n = m$.
\end{case}  
\begin{case} $r_u = r_{g-1} = 0$ and $r_x \neq 0$.  Since
\begin{align*}
v\cdot r_x = vx \cdot w_m(0) = (xv + (1-g) + xu)\cdot w_m(0) = 0,
\end{align*}
there exists $a\in \kut$ such that $r_x = a \zn{n}{0}$. We compute in two ways:
\begin{align*}
\zeta \cdot r_x &= \zt x \cdot w_m(0) = (x\zt - x)\cdot w_m(0) = m r_x + x\cdot s - 2 r_x = (m-2) r_x
\end{align*}
and also $\zt \cdot r_x = a \zt\cdot \zn{n}{0} = n r_x$. We conclude that $n = m-2$. \qed
\end{case}  

\subsection{Extensions of \texorpdfstring{$\smp{n}$}{} by \texorpdfstring{$\smp{n\pm 2}$}{}} Let $n\in \N_0$.
We next introduce  the $\D$-module $\usrl{n,1}$ generated by $z_n$ with defining  relations 
\begin{align}
\tag{\ref{eq:defn-rels-verma}}
&g\cdot z_n = z_n,& &\zeta \cdot z_n = -\tfrac{(n +2)}{2} z_n,  && u\cdot z_n = 0,  \qquad v\cdot z_n = 0,
\\ \label{eq:defn-rels-uniserial-}
&x^{2} \cdot z_n = 0,&  &y^{n +3}  \cdot z_n = 0, & &y^{n+1}x \cdot z_n = 0.
\end{align} 
This belongs to the family of $\D$-modules $\usrl{n,m}$ studied  in Section \ref{sec:hwr1}. 
The set
$\{\tz_{n}(i,j)\coloneqq y^i x^j \cdot z_n  \colon j\in\I_{0,1}, i\in\I_{0,n+2 - 2j}\}$ is a basis of $\usrl{n,1}$, see 
Section \ref{sec:hwr1} or prove it directly.
Given $b\in \kut$, let $\cE(b)$ be the exact sequence 
\begin{align*}
\xymatrix{ \smp{n}  \ar@{^{(}->}[rr] ^{\iota_b} & & \usrl{n,1} \ar@{->>}[rr] ^{\pi} & & \smp{n+2} },
\end{align*}
where $\iota_b$  and  $\pi$ are  determined by $\iota_b(\zn{n}{i})= b\tz_{n}(i,1)$ and $\pi(\tz_{n}(i,1)) = 0$ for $i\in\I_{0,n}$;
and   $\pi(\tz_{n}(k,0)) = z_{n+2}(k)$ for $k\in\I_{0,n+2}$.

\begin{prop}\label{prop:ext-Ln-Ln+2} 
Any extension $\xymatrix{ \smp{n}  \ar@{^{(}->}[r]  &  T \ar@{->>}[r]  &  \smp{n+2} }$ 
is either trivial or else isomorphic to $\cE(b)$ for a unique
$b\in \kut$.  Hence $T$ is isomorphic either
to $\smp{n}\oplus \smp{n+2}$ or to $\usrl{n,1}$ and
\begin{align} \label{eq:ext-Ln-Ln+2}
\dim \Ext^1_{\D}(\smp{n},\smp{n+2})= 1.
\end{align}
\end{prop}

\begin{proof} Identify $\smp{n}$ as a submodule of $T$ via $\iota$ and pick $\wb\in T$ satisfying 
\begin{align*}
\pi(\wb) = z_{n+2}(0)\in \smp{n+2}.
\end{align*}
Hence there exist $c_0,\dots, c_{n}\in\ku$ such that
\begin{align*}
\zeta\cdot \wb =(n+2)\wb + \sum_{i=0}^{n} c_i \zn{n}{i}.
\end{align*}
Let $w_{n+2}(0) \coloneqq  \wb + \sum_{i=0}^{n}\frac{c_i}{2(i+1)} \zn{n}{i}$. Clearly $\pi(w_0) = z_{n+2}(0)$, but also
\begin{align*}
\zt\cdot w_{n+2}(0) &= (n+2)\wb + \sum_{i=0}^{n}  c_i \zn{n}{i} + \sum_{i=0}^{n}\frac{c_i (n-2i)}{2(i+1)} \zn{n}{i} \\
&=(n+2)\wb + \sum_{i=0}^{n} \frac{(2i + 2 + n - 2i) c_i}{2(i+1)}  \zn{n}{i}\\
&=(n+2) \left(\wb + \sum_{i=0}^{n} \frac{c_i}{2(i+1)} \zn{n}{i}\right) = (n+2) w_{n+2}(0).
\end{align*}
Hence $w_{n+2}(0) \in T^{n+2}$, 
\begin{align*}
w_{n+2}(j) &\coloneqq y^j \cdot w_{n+2}(0) \in T^{n+2 - 2j}, &  \pi(w_{n+2}(j)) &= z_{n+2}(j), &
j &\in \I_{n+2},
\end{align*}
and $\{\zn{n}{0},\dots,\zn{n}{n},w_{n+2}(0),\dots,w_{n+2}(n+2)\}$ is a basis of $T$. Thus
\begin{align*}
T^{k} & = \begin{cases} \ku z_n(\tfrac{n-k}{2}) \oplus \ku w_{n+2}(\tfrac{n+2 -k}{2}),& k \in \I_{-n, n};
\\ \ku w_{n+2}(0) & k = n+2,
\\  \ku w_{n+2}(n+2) & k = -n - 2.
\end{cases} 
\end{align*}
Now $y^{n+3} \cdot w_{n+2}(0) \in T^{-n-4}$ and $u\cdot w_{n+2}(0), v\cdot w_{n+2}(0)  \in T^{n+4}$  by \eqref{eq:weight-action}, i.~e.
\begin{align}\label{eq:extension-verma1}
y^{n+3} \cdot w_{n+2}(0) &= 0, &   u \cdot w_{n+2}(0) & = 0, &   v \cdot w_{n+2}(0) & = 0.
\end{align}
Also $g \cdot w_{n+2}(0) \in T^{n+2}$, i.~e. $g \cdot w_{n+2}(0) = a w_{n+2}(0)$ for some $a\in \ku$; 
applying $\pi$ to both sides of the last equality  we get that $a=1$, that is
\begin{align}\label{eq:extension-verma2}
g \cdot w_{n+2}(0)  &= w_{n+2}(0).
\end{align}
Finally $x\cdot w_{n+2}(0) \in \smp{n} \cap T^{n} = \ku \zn{n}{0}$, i.~e. 
$x \cdot w_{n+2}(0) = c \zn{n}{0}$ for some $c\in \ku$. If $c= 0$, then the subspace with basis
$w_{n+2}(0),\dots,w_{n+2}(n+2)$ is a submodule isomorphic to $\smp{n+2}$ and the extension is trivial.
If $c \neq 0$, then the extension is isomorphic to $\cE(c^{-1})$. 
\end{proof}

\begin{prop}\label{prop:ext-Ln-Ln-2}
If $S$ is an extension of $\smp{n}$ by $\smp{n-2}$, then $S$ is isomorphic either
to $\smp{n}\oplus \smp{n-2}$ or to $\usrl{n-2,1}^*$.  Also, 
\begin{align} \label{eq:ext-Ln-Ln-2}
\dim \Ext^1_{\D}(\smp{n},\smp{n-2})= 1. 
\end{align}
\end{prop}
\pf
By Remark \ref{rem:extensions-dual} and Proposition \ref{prop:ext-Ln-Ln+2}.
\epf

\subsection{The quiver and representation type}
Propositions \ref{prop:Sn}, \ref{prop:ext-Ln-Ln+2} and \ref{prop:ext-Ln-Ln-2} give us 
the {\it Gabriel quiver} of $\D$, i.e.,  the quiver $\Ext Q(\D)$ with vertices $\N_0$ 
and $\dim \Ext^{1}_{\D}(S_i,S_j)$ arrows from the vertex $i$ to the vertex $j$. That is,

\begin{align}\label{eq:quiver}  
\begin{aligned}
\xymatrix@C=0.7cm{\underset{0} {\circ} \ar@/^0.5pc/[rr] \ar@(ul,ur) &&  
\underset{2} {\circ}\ar@/^0.5pc/[ll] \ar@/^0.5pc/[rr] \ar@(ul,ur)
&&  \underset{4} {\circ}\ar@/^0.5pc/[ll] \ar@(ul,ur)&\dots & \underset{2n} {\circ} \ar@/^0.5pc/[rr] \ar@(ul,ur)
&&  \underset{2n + 2} {\circ}\ar@/^0.5pc/[ll]  \ar@(ul,ur)\dots 
\\ 
\underset{1} {\circ} \ar@/^0.5pc/[rr] \ar@(ul,ur)&&  \underset{3} {\circ}\ar@/^0.5pc/[ll] \ar@/^0.5pc/[rr] \ar@(ul,ur)
&&  \underset{5} {\circ}\ar@/^0.5pc/[ll]  \ar@(ul,ur) &\dots & \underset{2n-1} {\circ} \ar@/^0.5pc/[rr] \ar@(ul,ur)
&&  \underset{2n + 1} {\circ}\ar@/^0.5pc/[ll] \ar@(ul,ur) \dots}
\end{aligned}
\end{align} 

From the analysis of this quiver one concludes:

\begin{prop}\label{prop:D-has-wild-rep-type}
The algebra $\D$ has wild representation type.
\end{prop}

\pf  This is evident for experts in representation theory of artin algebras but we include a proof for completeness.

\begin{step}\label{claim:injective-map-quivers} 
Let $A \twoheadrightarrow  B$ be a surjective map of algebras and $M, N \in \lmod{B}$. Then
the  canonical map $\Ext_B^1(M, N) \to  \Ext_A^1(M, N)$ is injective. 
\end{step}

\begin{step}\label{claim:bijective-map-quivers} 
Let now $A$ be a (possibly infinte-dimensional) algebra over a field $\ku$ with Ext-quiver $Q$  such thst 
$\dim_{\ku} \Ext_A^1(L, L') < \infty$ for any $L,  L' \in \Irr A$. 
Let $F$ be a finite subset of $\Irr A$ and let $Q_F$ be the  (full) subquiver of $Q$ spanned by $F$.
Then there exists a finite-dimensional quotient algebra $A \twoheadrightarrow  B$ such that the Ext-quiver of $B$ is 
isomorphic to $Q_F$.
\end{step}

Given $L,  L' \in F$, pick a basis $(v_i)$ of $\Ext_A^1(L, L')$ and for each $v_i$ an extension $M_i$  of $L$ by $L'$
representing $v_i$. 
Let $M$ be the direct sum of all $L, L'$ in $F$ and all the corresponding $M_i $. Clearly 
$\dim M< \infty$ hence so is the image $B$ of the representation $A \to \End M$. 
By construction and Claim \ref{claim:injective-map-quivers}  the  canonical map 
$\Ext_B^1(L, L') \to  \Ext_A^1(L, L')$ is bijective, hence the Ext-quiver of $B$ is isomorphic to $Q_F$.

By Claim \ref{claim:bijective-map-quivers} applied to $F = \{0,2,4\}$ there exists a surjective algebra map 
$\D \twoheadrightarrow  B$ where $\dim B < \infty$  and the Ext-quiver of $B$ is 
isomorphic to 

\begin{align}\label{eq:QF}
\xymatrix@C=0.7cm{\underset{0} {\circ} \ar@/^0.5pc/[rr] \ar@(ul,ur) &&  
\underset{2} {\circ}\ar@/^0.5pc/[ll]   \ar@(ul,ur)  \ar@/^0.5pc/[rr]&& \underset{4} {\circ} \ar@/^0.5pc/[ll] \ar@(lu,ru)}
\end{align}
Let $C$ be the basic algebra  which is Morita equivalent to $B$ and
 $\mathfrak r = \operatorname{rad} C$. Then $C/ \mathfrak r^2$
has finite or tame representation type if and only if  the separated quiver $\Gamma_s$ of \eqref{eq:QF} is a disjoint union
of Dynkin and affine Dynkin diagrams, see \cite[Theorem X.2.6]{ARS}. But $\Gamma_s$ has the form
\begin{align*}
\xymatrix@C=0.7cm{\underset{0} {\circ} \ar[rr] \ar[d]  &&  
\underset{2'} {\circ}     && \underset{4} {\circ} \ar[ll] \ar[d]
\\
\underset{0'} {\circ}   &&  
\underset{2} {\circ}\ar[ll] \ar[rr]   \ar[u]  && \underset{4'} {\circ} }
\end{align*}
Thus $C/ \mathfrak r^2$ has wild representation type, and \emph{a fortiori} $C$, $B$ and $\D$ also. 
\epf

\subsection*{Acknowledgements} N. A. thanks Hennig Krause and Bernhard Keller for enlightening discussions.

\end{document}